\documentclass[11pt, reqno]{amsart}
\usepackage{amssymb,latexsym,amsmath,amsfonts, amsthm,mathtools}
\usepackage[mathscr]{eucal}
\usepackage{xcolor}

\numberwithin{equation}{section}

\theoremstyle{definition}
\newtheorem{definition}{Definition}[section]

\newtheorem{remark}[definition]{Remark}

\theoremstyle{plain}
\newtheorem{theorem}[definition]{Theorem}
\newtheorem{lemma}[definition]{Lemma}
\newtheorem{proposition}[definition]{Proposition}

\newtheorem{result}[definition]{Result}

\newcommand\wtil[1]{\widetilde{#1}}

\newcommand{\Span}{\operatorname{span}}

\newcommand{\SH}{\operatorname{SH}}
\newcommand{\Frac}{\operatorname{Frac}}
\newcommand{\aut}{\operatorname{Aut}_{\mathrm{hol}}}
\newcommand{\Ker}{\operatorname{ker}}

\newcommand{\vfhol}{\mathfrak X_{\mathrm{hol}}}
\newcommand{\F}{\mathbb F}
\newcommand{\Kalmod}{\operatorname{Kalmod}}

\newcommand{\bdy}{\partial}

\newcommand{\holo}{\mathcal{O}}

\newcommand{\impl}{\Longrightarrow}

\newcommand{\cplx}{\mathbb{C}}

\newcommand{\rea}{\mathbb{R}}

\newcommand{\natu}{\mathbb{N}}

\begin{document}

\title[ a counterexample to the T\'{o}th-Varolin Conjecture]{A new class of smooth algebraic varieties with the density property and  a counterexample to the T\'{o}th-Varolin Conjecture}
\author{Sushil Gorai}
\address{Department of Mathematics and Statistics, Indian Institute of Science Education and Research Kolkata,
Mohanpur, Nadia, West Bengal 741246, India} 
\email{sushil.gorai@iiserkol.ac.in}


\begin{abstract}
We prove that the smooth affine hypersurfaces 
\[
X_{m,n}=\{\zeta^mz_0^n+ (1+z_1\zeta)z_2+\zeta z_3^2=1\}\subset\mathbb{C}^5,\;\; m,n\geq 2,
\]
have the density property. For every odd integer $d\geq 3$, we identify $X_{d-1,d}$ with a Kaliman modification along a Brieskorn variety of type $(d,2,2,2)$. We prove that the fourfolds $X_{d-1,d}$ are diffeomorphic to $\mathbb{C}^4$ but not biholomorphic to $\mathbb{C}^4$. These, therefore, give counterexamples to the T\'{o}th-Varolin conjecture. 
\end{abstract}

\maketitle

\section{Introduction}\label{sec-intro}
A Stein manifold $M$ has the \emph{density property} if the Lie algebra generated by its complete holomorphic vector fields is dense, in the compact-open topology, in the Lie algebra of all holomorphic vector fields on $M$.  This property is the infinitesimal form of the largeness of the holomorphic automorphism group that underlies Anders\'en--Lempert theory \cite{Andersen-Lempert1993}. This property was indroduced by Varolin \cite{Varolin2001} and studied by many 
researchers (see for instance \cite{Varolin1999, Varolin2000, TothVarolin2000, TothVarolin2006, KaKu2008, KaKu2008b, KaKu2015, Kutle2016, Leuenberger2016} and references therein). Kaliman and Kutzschebauch \cite{KaKu2008} gave a very strong criterion of algebraic density peoperty based on compatible pairs. The notion of a compatible pair was later generalized in \cite{Leuenberger2016} to give a criterion of density property. 

T\'oth and Varolin formulated the following recognition conjecture in their study of semisimple homogeneous spaces~\cite[Conjecture~1.4]{TothVarolin2006}:
\emph{If a Stein manifold is diffeomorphic to $\cplx^n$ and has the density property, then it is biholomorphic to $\cplx^n$.}

This recognition problem remains open; see  \cite[Problem~11.2]{ForstnericKutzschebauch2022} for a detailed discussion of the conjecture. A related problem is:

{\em Zaidenberg's problem} \cite{Zaid1999}: {\em Is affine variety biholomorphic to $\cplx^n$ also algebraically isomorphic to $\cplx^n$?}

The reason for mentioning this is a cubic threefold called the Koras-Russell cubic 
\[
X_{KR}=\{(x,y,z_0,z_1)\in\cplx^4: x^2y+x+z_0^2+z_1^3=0\}.
\]
It is diffeomorphic to $\cplx^3$, but it was shown not to be algebraically isomorphic to $\cplx^3$ \cite{ML1996}, i.e., $X_{KR}$ is an exotic $\cplx^3$.
A smooth affine algebraic variety is called an exotic affine space if it is diffeomorphic to $\cplx^n$ but not isomorphic to $\cplx^n$ as an algebraic variety. Ramanujam \cite{Ramanujam} proved that there are no exotic affine spaces in dimension two. See \cite{Zaid1999} for more details.

The Koras-Russell cubic and related varieties have been studied extensively.
Leuenberger~\cite{Leuenberger2016} showed that $X_{KR}$ has the density property, which made the question whether it is biholomorphic to $\cplx^3$ very interesting as positive or negative answer would solve an open question: T\'{o}th-Varolin conjecture lies in one side and Zaidenberg's problem lies in the other. It is still an open problem if $X_{KR}$ is biholomorphic to $\cplx^3$. Casals and Murphy \cite{CasMurphy2019} shown that the Koras-Russell cubic is Stein deformation-equivalent to $\cplx^3$. Hence, the corresponding Weinstein structures are also deformation equivalent. That leaves no bihomorphic invariant coming from symplectic geometry to distinguish $X_{KR}$ from $\cplx^3$
\smallskip

The aim of the present paper is to give an explicit family of examples that provides a negative answer to the above conjecture.
For integers $m,n\geq 2$, let
%\begin{equation}\label{eq:Fmn}
 \[
 F_{m,n}(\zeta,z_0,z_1,z_2,z_3)
 :=\zeta^mz_0^n+ (1+z_1\zeta)z_2+\zeta z_3^2-1
 \]
%\end{equation}
and define
\begin{equation}\label{eq-Xmn}
 X_{m,n}:=\{(\zeta,z)\in \cplx^5: F_{m,n}(\zeta,z)=0\}\subset\cplx^5.
\end{equation}

\begin{theorem}\label{thm-density}
For every pair of integers $m,n\geq 2$, the affine variety $X_{m,n}$ is a complex manifold of dimension four and has the holomorphic density property.  
\end{theorem}

\begin{remark}
    The idea of the class of affine varieties as in \eqref{eq-Xmn} comes from looking at the Kaliman modification along a Brieskorn variety in McLean's paper \cite{McLean2009}, where he proved certain bound on the idempotent elements of symplectic homology, in particular, symplectic homology is nonzero. The exact form of Kaliman modification along certain Brieskorn varieties are given in the next theorem. A bihomorphism between two finite Stein manifolds (see Section~\ref{sec-prelims} for a definition) preserve the corresponding symplectic properties, i.e., the Weinstein structures arising from those Stein structure (see \cite{CiEli2012} for details about Weinstein structures). This was the crucial idea that led us to investigate the above class of affine varieties.
\end{remark}
The geometric and symplectic specialization is as follows.

\begin{theorem}\label{thm-identification}
For every integer $d\geq 3$, let
 $P_d(r,a,b,c)=r^d+ab+c^2$,
 $Z_d=P^{-1}(0)\subset\cplx^4$, and $p=(0,1,0,0)\in Z_d$.
Then, $X_{d-1,d}$ is isomorphic, as an affine variety, to the Kaliman modification
\[
 \mathcal{K}_4^d:=\Kalmod(\cplx^4,Z_d,\{p\}).
\]
\end{theorem}
After the linear change $a=u+iv$, $b=u-iv$, for $d=7$, this is the affine fourfold $K_4$ appearing in McLean's construction~\cite[Section~3.1]{McLean2009} (see Subsection~\ref{ss-McLean} for a brief description and Section~\ref{sec-Kaliman} for more details about Kaliman modification in our context).
Combining Theorems~\ref{thm-density} and~\ref{thm-identification} with McLean's results gives us a counterexample of the above conjecture.

\begin{theorem}\label{thm-main}
The smooth affine fourfold
\begin{equation}\label{eq-counter}
     X:=X_{6,7}
 =\left\{\zeta^6z_0^7+(1+z_1\zeta)z_2+\zeta z_3^2=1\right\}
\end{equation}
in $\cplx^5$ has the following properties:
\begin{itemize}
\item [(i)] $X$ is diffeomorphic to $\cplx^4$;
\item[(ii)] $X$ has the density property;
\item [(iii)] $X$ is not biholomorphic to $\cplx^4$.
\end{itemize}
\end{theorem}

Actually, there is a family of smooth algebraic varieties of which each of them is a counterexample to the T\'{o}th-Varolin conjecture. Here we present:

\begin{theorem}\label{thm-counterfamily}
    For $d\geq 2$ odd positive integer, the fourfold $X_{d-1,d}$ satisfies (i), (ii) and (iii) of Theorem~\ref{thm-main}.
\end{theorem}

Some comments on the layout of the paper: Section~\ref{sec-prelims} collects some results that will be used, a brief description of Kaliman modification and biholomorphic invariance of symplectic homology for certain Stein manifolds. Section~\ref{sec-family} is short and describes the smoothness and irreducibility of $X_{m,n}$. We prove the density property for $X_{m,n}$ in Section~\ref{sec-density}. 
In section~\ref{sec-Kaliman}, Theorem~\ref{thm-identification} is proved. Section~\ref{sec-symplectic} gives the complete proof of Theorem~\ref{thm-main}.

\section{Technical Preliminaries}\label{sec-prelims}

\subsection{The density property}
Let $M$ be a complex manifold.  We denote by $\holo(M)$ the space of homorphic function on $M$ and let $\vfhol(M)$ denote 
the Fr\'echet Lie algebra of holomorphic vector fields on $M$, with the compact-open topology on coefficients in local charts.

\begin{definition}
A holomorphic vector field $V$ on $M$ is \emph{complete} if its holomorphic flow
\[
 \Phi^V:\cplx\times M\longrightarrow M
\]
is defined for every complex time.  The manifold $M$ has the \emph{holomorphic density property} if
 the Lie algebra generated by complete holomorphic vector fields on $M$ is dense in $\vfhol(M)$.

\end{definition}

The main tool for showing density property in this paper is a theorem by Leuenberger\cite{Leuenberger2016}. We now recall the holomorphic formulation used in that paper~\cite[Definitions~2.4, 2.8 and 2.10]{Leuenberger2016}.

\begin{definition}\label{def:generating}
Fix $p\in M$. We call a subset $S\subset T_pM$ a \emph{generating set} if the orbit of $S$ under the differential action of the isotropy subgroup
\[
 \aut(M)_p:=\{\Phi\in\aut(M):\Phi(p)=p\}
\]
contains a basis of $T_pM$.
\end{definition}

\begin{definition}\label{def:compatible}
Let $\nu$ and $\mu$ be complete holomorphic vector fields on a Stein manifold $M$.
 The pair $(\nu,\mu)$ is \emph{semicompatible} if the closure of
\[
 \Span\left\{(\Ker\nu)(\Ker\mu)\right\}
\]
in $\holo(M)$ in the compact-open topology contains a nonzero ideal $I\subset\holo(M)$.
A semicompatible pair is \emph{compatible} if there exists $h\in\holo(M)$ such that
\[
 h\in\Ker\mu,
 \qquad
 0\neq\nu(h)\in\Ker\nu.
\]

\end{definition}

The following result by Leuenberger \cite[Theorem~1.2(1)]{Leuenberger2016} (see also 
\cite{KaKu2015}) will be vital for proving density property for the hypersurfaces $X_{m,n}$. 

\begin{result}\label{crit:leuenberger}
Let $M$ be a Stein manifold on which $\aut(M)$ acts transitively.  Suppose there are compatible pairs $(\nu_j,\mu_j)$ and a point $p\in M$ such that the vectors $\mu_j[p]$ form a generating set of $T_pM$.  Then $M$ has the holomorphic density property.
\end{result}

For completeness, the mechanism behind Result~\ref{crit:leuenberger}  is worth recalling.  If $I_j$ is the ideal associated with a compatible pair and $h_j$ is the compatibility function, then Leuenberger's commutator identity places the $\holo(M)$-module generated by
\[
 I_j\,\nu_j(h_j)\,\mu_j
\]
in the closure of the Lie algebra generated by complete fields.  Transitivity and the generating condition then promote these modules to all holomorphic vector fields by Cartan's theorems and a Nakayama-type argument.

A standard orbit argument will be used repeatedly.

\begin{lemma}\label{lem:orbit}
Let $M$ be a connected complex manifold and let $\mathscr V$ be a collection of complete holomorphic vector fields.  Let $G\subset\aut(M)$ be the subgroup generated by their flows.  If the evaluations at every point of the Lie algebra generated by $\mathscr V$ span the full tangent space, then every $G$-orbit is open.  Consequently $G$ acts transitively on $M$.
\end{lemma}

\begin{proof}
The tangent space to a $G$-orbit contains the evaluations of the original fields and is stable under Lie brackets; equivalently, the orbit distribution contains the evaluation of the Lie algebra they generate.  Full rank therefore makes every orbit an immersed open submanifold of $M$.  Distinct orbits are disjoint open subsets.  Connectedness of $M$ forces the existence of only one orbit.
\end{proof}

The follwing couple of lemmas are well known.
\begin{lemma}\label{lem-replica}
   If $V$ is complete and $f\in\Ker V$, then $fV$ is complete, with flow
\begin{equation}\label{eq:replica-flow}
 \Phi_t^{fV}(p)=\Phi_{t f(p)}^V(p).
\end{equation} 
\end{lemma}

The following elementary linearization will be used to verify the generating condition.

\begin{lemma}\label{lem:replica-derivative}
Let $V$ be a complete holomorphic vector field, let $f\in\Ker V$, and let $p\in M$ satisfy $f(p)=0$.  Then the flow of $fV$ fixes $p$, and
\begin{equation}\label{eq:replica-derivative}
 d\Phi_t^{fV}\big|_p(v)
 =v+t\,df_p(v)V|_p,
 \qquad v\in T_pM.
\end{equation}
\end{lemma}

\begin{proof}
Since $f(p)=0$, $(fV)|_p=0$. The point $p$ remains fixed under the flows of $fV$.
The differential of $fV$ at $p$ is the rank-one map
\[
 N(v):=df_p(v)V|_p.
\]
Since $V(f)=0$, we get that $N^2=0$.  The differential of the flow at the fixed point is therefore
\[
 \exp(tN)=I+tN,
\]
which is exactly~\eqref{eq:replica-derivative}.
\end{proof}

Next, we mention a simple lemma from linear algebra. This will be used in our proof of Theorem~\ref{thm-density}.
\begin{lemma}\label{lem:koszul}
Let $(a,b,c)\in\cplx^3$ be a nonzero vector.  The vectors
\begin{equation}\label{eq:koszul-vectors}
 u_1=(-c,0,a),
 \qquad
 u_2=(-b,a,0),
 \qquad
 u_3=(0,-c,b)
\end{equation}
lie in the kernel of the linear functional
\[
 (\xi,\eta,\zeta)\longmapsto a\xi+b\eta+c\zeta
\]
and span that two-dimensional kernel.
\end{lemma}

\begin{proof}
Clearly the vectors lie in the kernel of the linear functional $(\xi,\eta,\zeta)\longmapsto a\xi+b\eta+c\zeta$. 
If $a\neq0$, then $u_1$ and $u_2$ are linearly independent. If $a=0$, then $(b,c)\neq(0,0)$.  The vector $u_3=(0,-c,b)$ is nonzero, and one of $u_1=(-c,0,0)$ or $u_2=(-b,0,0)$ supplies the other independent vector so that the span of $u_1,u_2$ and $u_3$ becomes two dimensional subspace. 
\end{proof}

Following lemma is quite well known.
\begin{lemma}\label{lem:polynomial-dense}
The coordinate ring $\cplx[X_{m,n}]$ is dense in $\holo(X_{m,n})$ in the compact-open topology.
\end{lemma}

\begin{proof}
The variety $X:=X_{m,n}$ is a closed complex submanifold of $\cplx^5$.  By Cartan's Theorem~B applied to the exact sequence
\[
 0\longrightarrow\mathcal I_X
 \longrightarrow\holo_{\cplx^5}
 \longrightarrow\holo_X
 \longrightarrow0,
\]
every holomorphic function on $X$ extends to an entire function on $\cplx^5$.  The Taylor polynomials of the extension converge uniformly on compact subsets of $\cplx^5$, hence their restrictions converge uniformly on compact subsets of $X$.  These restrictions lie in $\cplx[X]$.
\end{proof}

\subsection{Kaliman modifications}
This notion was introduced by Kaliman \cite{Kaliman1994} in 1994. It produced many exotic algebraic varieties.
Consider complex algebraic subvarieties $C$ and $D$ of a smooth algebraic variety $M$ such that $C$ is smooth, $C\subset D$ and $C$ is irreducible. Assume $\dim C<\dim D$. Then the Kaliman modification of the triple $(M,D,C)$ is defined as 
\[
\Kalmod (M,D,C)=\wtil{M}\setminus\wtil{D},
\]
where $\wtil{M}$ is the blow up of $M$ along the variety $C$ and $\wtil{D}$ is the strict/proper transform of $D$ in $\wtil{M}$. In algebraic terms:
let $A$ be a finitely generated $\cplx$-algebra, let $f\in A$ define an irreducible divisor $D$, and let $I\subset A$ be an ideal containing $f$, defining a center $C\subset D$.  The Kaliman modification has coordinate ring
\begin{equation}\label{eq:affine-modification-ring}
 A[I/f]
 :=A\left[\frac{g}{f}:g\in I\right]
 \subset A[f^{-1}].
\end{equation}

The following couple of results from \cite{Kaliman1994} will be very useful for us.
\begin{result}[Kaliman]\label{res-Kaliman}
    The Kaliman modification of an affine variety is an affine variety.
\end{result}

\begin{result}\cite[Theorem~3.5]{Kaliman1994}\label{res-kalmodcontractible}
    Let $M$ be a smooth contractible affine algebraic variety of dimension $\geq 2$ and $D$ be an irreducible hypersurface such that either $D$ is smooth contractible or $D$ is homeomorphic to $\rea^{2k}$ for some $k\in\natu$. Suppose $C$ is a contractible subvariety of $M$ and $C\subset D_{reg}$. Then $\Kalmod(M,D,C)$ is contractible.
\end{result}

\subsection{Finite-type Stein structures and biholomorphic invariance of symplectic homology}
\begin{definition}
    A Stein manifold is said to be finite-type if it admits an exhausting strictly plurisubharmonic Morse function with finitely many critical points. 
\end{definition}
A finite-type Stein manifold is denoted by $(X,J,\phi)$, where $(X,J)$ is a Stein manifold and $\phi$ is strictly plurisubharmonic Morse function with finitely many critical point such that 
for all $c\in\rea$, sublevel set $\{z\in X: \phi(z)\leq c\}$ is compact.
 Clearly, sufficiently large regular sublevel sets of the strictly plurisubharmonic Morse function are Liouville domains, and their completions determine a canonical exact symplectic manifold up to symplectomorphism.  McLean records the following facts~\cite[Section~2.1]{McLean2012}:
\begin{itemize}
\item any two complete finite-type strictly plurisubharmonic exhaustions on the same Stein manifold determine symplectomorphic completions;
\item the resulting symplectic form is a biholomorphic invariant;
\item every smooth affine variety is a finite-type Stein manifold.
\end{itemize}
Consequently, symplectic homology and its ring structure can be used as biholomorphic invariants for smooth affine varieties.

We shall also use the standard vanishing theorem for subcritical Stein manifolds: their symplectic homology vanishes~\cite{Cieliebak2002}. Since the function  $z\to |z|^2$ has only one critical point of Morse index zero, in particular, we obtain that
\begin{equation}\label{eq:SH-C4}
 \SH(\cplx^4;\F_2)=0,
\end{equation}
.

The biholomorphic invariance of symplectic homology rings---the symplectic structures induce from the underlined Stein structures---plays crucial role in our proof.
The following proposition isolate the exact implication needed for our argument.

\begin{proposition}\label{prop:biholo-SH-invariance}
Let $M$ and $N$ be finite-type Stein manifolds.  If $M$ and $N$ are biholomorphic, then their canonical completed exact symplectic manifolds are symplectomorphic.  Consequently their symplectic-homology rings are isomorphic.
\end{proposition}

\begin{proof}
Choose a complete finite-type exhausting strictly plurisubharmonic function $\phi_N$ on $N$.  If $\Phi:M\to N$ is biholomorphic, then $\phi_N\circ\Phi$ is a complete finite-type exhausting strictly plurisubharmonic function on $M$, and $\Phi$ identifies the associated exact symplectic structures.

McLean's uniqueness statement says that any two complete finite-type exhausting strictly plurisubharmonic functions on the same Stein manifold yield symplectomorphic Stein completions~\cite[Section~2.1]{McLean2012}.  Thus the structure pulled back from $N$ is symplectomorphic to the canonical one attached to any chosen finite-type exhaustion on $M$.  Symplectic homology is invariant under this exact symplectic equivalence.
\end{proof}

\subsection{McLean's construction.}\label{ss-McLean} In this subsection we briefly recall McLean's construction of exotic Stein manifolds using Kaliman modification. McLean considered the Brieskorn hypersurface
\[
 Z=\{r^7+u^2+v^2+s^2=0\}\subset\cplx^4. 
\]
Let $p\in Z\setminus \{0\}$. The space that McLean considered is the Kaliman modification of the triple $(\cplx^4, Z,\{p\})$, i,e., 
\[
K_4:=\Kalmod(\cplx^4, Z,\{p\}).
\]
The topology, specifically symplectic homology, plays very crucial role in our proof of Theorem~\ref{thm-main}. Some details about the topology that will used for our proof are given in Section~\ref{sec-symplectic}.

\section{The affine family $X_{m,n}$: irreducibility and smoothness}\label{sec-family}

First we show the irreducibility of $F_{m,n}$ in $\cplx[\zeta, z_0,z_1,z_2,z_3]$.
\begin{proposition}\label{prop:irreducible}
The polynomial $F_{m,n}$ is irreducible in $\cplx[\zeta,z_0,z_1,z_2,z_3]$.  Hence, $X_{m,n}$ is an irreducible affine variety of complex dimension four.
\end{proposition}

\begin{proof}
Consider the unique factorization domain
$A:=\cplx[\zeta,z_0,z_1,z_3]$ and
\[
 F_{m,n}(\zeta,z_0,z_1,z_2,z_3)
 :=(1+z_1\zeta)z_2+ (\zeta^mz_0^n+ \zeta z_3^2-1)
\]
as a polynomial of degree one in $z_2$ over 
$A$.
The coefficient $1+z_1\zeta$ of $z_2$ is irreducible in $A$.  It cannot divide the second coefficient, which is a nonzero polynomial independent of $z_1$.  Thus, the two coefficients are coprime.  Hence, the polynomial $F_{m,n}$ is primitive in $A[z_2]$ and linear, hence, irreducible over $\Frac(A)$.  Gauss's lemma gives us the irreducibility of $F_{m,n}$ in $A[z_2]$.
The quotient coordinate ring is therefore a domain, and the nonzero affine variety $X$ in $\cplx^5$ has dimension four.
\end{proof}

Next, we show that $X_{m,n}$ is a smooth complex hypersurface. 

\begin{proposition}\label{prop:smooth-general}
The affine variety $X_{m,n}$ is smooth.  In particular, it is a connected Stein manifold.
\end{proposition}

\begin{proof}
Let $(\zeta,z_0,z_1,z_2,z_3)\in X_{m,n}$.
    This implies that 
    \[
 F_{m,n}(\zeta,z_0,z_1,z_2,z_3)
 :=(1+z_1\zeta)z_2+ (\zeta^mz_0^n+ \zeta z_3^2-1).
\]
Clearly, 
\begin{equation}\label{eq-derivative}
  \dfrac{\bdy F_{m,n}}{\bdy z_2}=1+z_1\zeta.  
\end{equation}
Hence,  $X_{m,n}$ is smooth at every point of the set $\{(\zeta,z)\in X_{m,n}: 1+z_1\zeta\neq 0\}$. We further calculate 
\begin{align*}
\dfrac{\bdy F_{m,n}}{\bdy z_1} &=\zeta z_2\\
\dfrac{\bdy F_{m,n}}{\bdy z_0} &=n\zeta^mz_0^{n-1}\\
\dfrac{\bdy F_{m,n}}{\bdy z_3} &=2\zeta z_3
\end{align*}
From \eqref{eq-derivative} along with the above three equations we obtain that 
\[
X_{m,n}\cap\{(\zeta,z)\in\cplx^5: 1+z_1\zeta=0, \zeta z_2=0, \zeta^m z_0^{n-1}=0, \zeta z_3=0\}=\varnothing.
\]
Hence, $X_{m,n}$ is a smooth hypersurface for every $m,n\in N$
\end{proof}

%  DENSITY PROPERTY

\section{The density property of $X_{m,n}$}\label{sec-density}

Fix integers $m, n\geq 2$ throughout this section and denote $X:=X_{m,n}$. Throughout this section we will use this notation.
Our main tool to show the density property of $X$ is Result~\ref{crit:leuenberger}.  For that, we need to show that there are enough complete vector fields such that 
\begin{itemize}
    \item[i)] $\aut X$ acts transitively on $X$. 
    \item[ii)] There exist compatible pair $(\nu,\mu)$ as in the definition \ref{def:compatible}.
    \item [iii)] There exists a point $p_0$ such that the set $\{\mu|_{p_0}: (\nu,\mu)\;\text{is a compatible pair}\;\;\}$ generates the tangent space $T_{p_0}X$ in the sense of defintition~\ref{def:generating}.
\end{itemize}
To show these we break the proof in a few parts
\bigskip

{\bf \em Part I. Existence of some complete holomorphic vector fields.}
\smallskip

 \noindent We first consider the following vector fields of $\cplx^5$ constructed using the defining function $F_{m,n}$.
\begin{align*}
V_1 &: = 2z_3\dfrac{\bdy}{\bdy z_1} -z_2\dfrac{\bdy}{\bdy z_3}\\
V_2 &:= (1+z_1\zeta)\dfrac{\bdy}{\bdy z_3}-2z_3\zeta \dfrac{\bdy}{\bdy z_2}\\
V_3 &:= nz_0^{n-1}\zeta^{m-1}\dfrac{\bdy}{\bdy z_1}-z_2\dfrac{\bdy}{\bdy z_0}\\
V_4 &:= (1+z_1\zeta)\dfrac{\bdy}{\bdy z_0} - nz_0^{n-1}\zeta^m\dfrac{\bdy}{\bdy z_2}\\
V_5 &:= (1+z_1\zeta)\dfrac{\bdy}{\bdy \zeta} - (mz_0^n \zeta^{m-1}+z_1z_2+z_3^2)\dfrac{\bdy}{\bdy z_2}.
\end{align*}

Clearly, we obtain that 
\[
V_j(F_{m,n})=0\;\;\forall j=1,2,3,4,5.
\]
Hence, $V_j\in \vfhol(X)$. We now show that each of the above vector fields are complete. 
For that, we first compute:
\begin{align*}
  & V_1(\zeta) =0=V_1(z_0)=V_1(z_2)  \\
  & V_1(z_1) =2z_3, V_1(z_3)=z_2\\
  \impl & V_1^3(z_1)=0=V_1^2(z_3).
\end{align*}
Similarly, we 
also obtain that 
\begin{align*}
V_2(\zeta)& =V_2(z_0)=V_2(z_1)=0=V_2^2(z_3)=V_2^3(z_2)\\
V_3(\zeta)&=V_3(z_2)=V_3(z_3)=0=V_3^2(z_2)=V_3^{n+1}\\
V_4(\zeta)&=V_4(z_1)=V_4(z_3)=0=V_4^2(z_0)=V_4^{n+1}.
\end{align*}
Therefore, 
$V_1, V_2,V_3$ and $V_4$ define locally  nilpotent derivations of the polynomial ring $\cplx[\zeta,z_0,z_1,z_2,z_3]$. Their exponentials give polynomial flows defined for every complex time. Hence, $V_j$, $j=1,2,3,4$, are $\cplx$-complete. Since $V_5$ is not nilpotent, we will show its completeness in the following proposition.

\begin{proposition}
Let \(m\geq 1\) and \(n\geq 0\) be integers. The polynomial
vector field
\[
V_5=(1+z_1\zeta)\frac{\partial}{\partial \zeta}
-\left(mz_0^n\zeta^{m-1}+z_1z_2+z_3^2\right)
 \frac{\partial}{\partial z_2}
\]
is $\cplx$-complete on $\cplx^5$, with coordinates
\((\zeta,z_0,z_1,z_2,z_3)\). 
\end{proposition}

\begin{proof}
Along an integral curve, the coordinates \(z_0,z_1,z_3\)
remain constant. To simplify the flow equations a bit, we put
\[
a=z_1,\qquad b=z_0^n,\qquad c=z_3^2,
\]
and prescribe the initial conditions
\(\zeta(0)=\zeta_0\) and \(z_2(0)=\eta\).
The flow equations are
\[
\zeta'(t)=1+a\zeta(t),\qquad
z_2'(t)+az_2(t)=-mb\zeta(t)^{m-1}-c.
\]

Define
\[
E_a(t)=
\begin{cases}
\dfrac{e^{at}-1}{a},&a\neq 0,\\[6pt]
t,&a=0.
\end{cases}
\]
This function is jointly entire in \((a,t)\). 
Solving the first equation gives us
\[
T_t:=\zeta(t)=e^{at}\zeta_0+E_a(t).
\]
Using the integrating factor \(e^{at}\) in the second
equation gives
\[
W_t:=z_2(t)=e^{-at}
\left[
\eta-\int_0^t e^{au}
\left(mbT_u^{m-1}+c\right)\,du
\right].
\]
The integrand is entire in \(u\), so the integral is
path-independent and exists for every \(t\in\cplx\).
Integrating along \(u=vt\), \(0\leq v\leq 1\), also shows
that \(W_t\) depends holomorphically on \(t\) and all
the initial coordinates.

For a formula without apparent singularities, observe that
\[
T_t-\zeta_0=(1+a\zeta_0)E_a(t),\qquad
\frac{dT_t}{dt}=(1+a\zeta_0)e^{at}.
\]
When \(1+a\zeta_0\neq 0\), these identities imply
\[
\int_0^t me^{au}T_u^{m-1}\,du
=\frac{T_t^m-\zeta_0^m}{1+a\zeta_0}
=E_a(t)\sum_{j=0}^{m-1}T_t^{m-1-j}\zeta_0^j.
\]
Both sides are holomorphic in \((a,\zeta_0,t)\), so the
last identity extends across \(1+a\zeta_0=0\).
Consequently,
\[
W_t=e^{-at}\left[
\eta-E_a(t)\left(
c+b\sum_{j=0}^{m-1}T_t^{m-1-j}\zeta_0^j
\right)
\right].
\]
%In particular, when \(a=0\),
%\[
%T_s=t_0+s,\qquad
%W_s=\eta-b\bigl((t_0+s)^m-t_0^m\bigr)-cs.
%\]

Hence,
\[
\Phi_t(\zeta_0,z_0,z_1,\eta,z_3)
=(T_t,z_0,z_1,W_t,z_3)
\]
defines a jointly entire flow for every complex time
\(t\in\cplx\). 
Therefore,  $V_5$ is $\cplx$-complete.
\end{proof}

{\bf \em Part II. $\aut X$ acts transitively on $X$.}
\smallskip

\noindent We now show the transitivity of the group $\aut X$. 
First, we mention some commutator identities on $X$ that will be useful in our proofs.
\begin{align}
 [V_1,V_2]
 &=-2(1+z_1\zeta)\frac{\bdy}{\bdy z_1}+2z_2\zeta\frac{\bdy}{\bdy z_2},
 \label{eq:bracket-12}\\
[V_1,V_4]
 &=2z_3\zeta\frac{\bdy}{\bdy z_0}-n\zeta^mz_0^{n-1}\frac{\bdy}{\bdy z_3},
 \label{eq:bracket-14}\\
 [V_1,V_5]
 &=2z_3\zeta\frac{\bdy}{\bdy \zeta}-(mz_0^n \zeta^{m-1}+z_1z_2+z_3^2)\frac{\bdy}{\bdy z_3},
 \label{eq:bracket-15}\\
 [V_3,V_5]
 &=nz_0^{n-1}\zeta^m\frac{\bdy}{\bdy \zeta}-(mz_0^n \zeta^{m-1}+z_1z_2+z_3^2)\frac{\bdy}{\bdy z_0} \notag \\
& -n(m-1)\zeta^{m-2}z_0^{n-1}(1+z_1\zeta)\frac{\bdy}{\bdy z_1} 
 +n(m-1)\zeta^{m-1}z_0^{n-1}z_2\frac{\bdy}{\bdy z_2}.
 \label{eq:bracket-35}
\end{align}

We first show  here that the Lie algebra generated by the vector fields $V_j$, $j=1,\dots, 5$, spans $T_pX$ for every $p\in X$. Suppose a point 
$p\in X\cap \{1+z_1\zeta\neq0\}$. We treat this case in the following proposition.

\begin{proposition}\label{prop:span-q-nonzero}
At every point $p=(s, w_0,w_1,w_2.w_3)\in X$ with $1+sw_1\neq0$, the four vectors
$ V_2|_p, V_4|_p, V_5|_p, [V_1,V_2]_p$
form a basis of $T_pX$.
\end{proposition}

\begin{proof}
Project the four vectors onto the coordinate directions
\[
 \frac{\bdy}{\bdy \zeta},\quad
 \frac{\bdy}{\bdy z_0},\quad
 \frac{\bdy}{\bdy z_3},\quad
 \frac{\bdy}{\bdy z_1}.
\]
Their projected components are
\[
 (1+sw_1)\frac{\bdy}{\bdy z_3}|_p,
 \qquad
 (1+sw_1)\frac{\bdy}{\bdy z_0}|_p,
 \qquad
-(1+sw_1)\frac{\bdy}{\bdy \zeta}|_p,
 \qquad
 2(1+sw_1)\frac{\bdy}{\bdy z_1}|_p
\]
respectively.
%with only additional $\dfrac{\bdy}{\bdy z_2}$ components.  
Since $1+sw_1\neq0$, the projected vectors are linearly independent.  As $T_pX$ has dimension four, they form a basis.
\end{proof}

In the next case, suppose the point $p$ lies in 
the hypersurface $\{1+z_1\zeta=0\}$.
Suppose now that $p=(\zeta, z_0,z_1,z_2,z_3)\in X$ and $1+\zeta z_1=0$.  Then, $\zeta\neq0$, and the defining equation reduces to
\begin{equation}\label{eq:zero-equation}
\zeta^mz_0^n+ \zeta z_3^2-1=0
\end{equation}
Hence, $z_0$ and $z_3$ cannot both vanish.  It follows that
\begin{equation}\label{eq:a-b-not-both-zero}
 (nz_0^{n-1}\zeta^m, 2z_3\zeta)\neq(0,0).
\end{equation}
Indeed, since $\zeta\neq 0$, $2z_3\zeta=0$ implies that $z_3=0$, hence, \eqref{eq:zero-equation} gives $z_0\neq0$.
%and $nz_0^{n-1}t^m\neq0$.
At such a point,
\[
 V_4|_p=-nz_0^{n-1}\zeta^m\frac{\bdy}{\bdy z_2},
 \qquad
 V_2|_p=-2z_3\zeta\frac{\bdy}{\bdy z_2}.
\]
Thus, by~\eqref{eq:a-b-not-both-zero}, the Lie algebra generated by $V_j, j=1,\dots, 5$ contain the $\dfrac{\bdy}{\bdy z_2}|_p$.

\begin{proposition}\label{prop:span-qzero-y-nonzero}
Let $p=(\zeta,z_0,z_1,z_2,z_3)\in X$ satisfy $1+z_1\zeta=0$ and $z_2\neq0$.  Then the Lie algebra generated by $V_j|_p$, $j=1,\dots, 5$, spans $T_pX$.
\end{proposition}

\begin{proof}
The argument before this proposition gives us that the vector $\dfrac{\bdy}{\bdy z_2}|_p$ is contained in the Lie algebra generated by $V_j|_p$, $j=1,\dots, 5$.

Assume first that $z_3\neq0$.  
%Modulo the $\dfrac{\bdy}{\bdy z_2}$-direction, 
Consider
\[
 V_1|_p,\qquad V_3|_p,\qquad [V_1,V_5]|_p.
\]
Their projections onto the subspace spanned by $\frac{\bdy}{\bdy \zeta},
 \frac{\bdy}{\bdy z_0}$, and
 $\frac{\bdy}{\bdy z_3}$, after identifying with
$(t,z_0,z_3)$-coordinates are
\[
 (0,0,-z_2),
 \qquad
 (0,-z_2,0),
 \qquad
 (-2z_3\zeta,0,mz_0^n \zeta^{m-1}+z_1z_2+z_3^2).
\]
These vectors are linearly independent.  Hence, $V_1|_p,V_3|_p, [V_1,V_5]|_p$, together with $\dfrac{\bdy}{\bdy z_2}|_p$, span the four-dimensional tangent space $T_pX$.

Assume next that $z_3=0$.  Then $z_0^{n-1}\zeta^m\neq0$.  
%Modulo $\dfrac{\bdy}{\bdy z_2}$, 
Here we use
\[
 V_1|_p,\qquad V_3|_p,\qquad [V_3,V_5]|_p.
\]
As before, their projections to $(t,z_0,z_3)$ are
\[
 (0,0,-z_2),
 \qquad
 (0,-z_2,0),
 \qquad
 (nz_0^{n-1}\zeta^m, -(mz_0^n \zeta^{m-1}+z_1z_2+z_3^2),0).
\]
These vectors are again linearly independent.  Hence, the Lie algebra generated by $V_j|_p$, $j=1,\dots, 5$, spans $T_pX$.
\end{proof}

\begin{proposition}\label{prop:span-qzero-yzero}
Let $p=(\zeta,z_0,z_1,0,z_3)\in X$ satisfying $1+z_1\zeta=0$. Then, the Lie algebra generated by $V_j$, $j=1,\dots, 5$, spans $T_pX$.
\end{proposition}

\begin{proof}
By the preceding argument, $\dfrac{\bdy}{\bdy z_2}|_p$ lies in the Lie algebra generated by $V_j$, $j=1,\dots, 5$.  Consider,
\[
 V_1|_p=2z_3\dfrac{\bdy}{\bdy z_1}|_p,
 \qquad
 V_3|_p=n\zeta^{m-1}z_0^{n-1}\dfrac{\bdy}{\bdy z_1}|_p.
\]
Thanks to \eqref{eq:a-b-not-both-zero}, at least one of $z_0$ and $z_3$ is nonzero.  If $z_3\neq0$, then $V_1|_p$ gives us that $\bdy_{z_1}|_p$ lies in the Lie algebra generated by $V_j$, $j=1,\dots, 5$.  If $z_3=0$, then $z_0\neq0$ and $n \zeta^{m-1}z_0^{n-1}\neq0$. Then, $V_3|_p$ supplies $\bdy_{z_1}|_p$ in the Lie algebra generated by $V_j$, $j=1,\dots, 5$.
%It spans the two-dimensional tangent subspace in the $(t,z_0,z_3)$-directions.  
Since $1+z_1\zeta=z_2=0$, the differential of the defining equation restricted to these directions is
\begin{equation}\label{eq:tangent-row}
 (mz_0^n \zeta^{m-1}+z_1z_2+z_3^2)\,\alpha+n\zeta^mz_0^{n-1}\,\beta+2z_3\zeta\,\gamma=0.
\end{equation}
The projections of the three brackets~\eqref{eq:bracket-15}, \eqref{eq:bracket-35}, and~\eqref{eq:bracket-14} onto $(\zeta,z_0,z_3)$ are respectively
$(-2z_3\zeta,0,mz_0^n \zeta^{m-1}+z_1z_2+z_3^2)$,
 $(-n\zeta^{m-1}z_0^{n-1}, mz_0^n \zeta^{m-1}+z_1z_2+z_3^2,0)$ and 
$(0,-2z_3\zeta,n\zeta^{m-1}z_0^{n-1}).$
By Lemma~\ref{lem:koszul}, these vectors span the kernel of~\eqref{eq:tangent-row}.  Together with the $\dfrac{\bdy}{\bdy z_1}$ and $\dfrac{\bdy}{\bdy z_2}$ directions, they span $T_pX$.
\end{proof}

\begin{theorem}\label{thm:transitivity}
The subgroup of $\aut X$ generated by the complete flows of $V_j$, $j=1,\dots, 5$, acts transitively on $X$.
\end{theorem}

\begin{proof}
Proposition~\ref{prop:span-q-nonzero} gives full tangent rank on $\{1+z_1\zeta\neq0\}$.  Propositions~\ref{prop:span-qzero-y-nonzero} and~\ref{prop:span-qzero-yzero} give full rank on $\{1+z_1\zeta=0\}$.  Therefore, the Lie algebra generated by the five complete vector fields spans every tangent space.  Since $X$ is connected, by Proposition~\ref{prop:smooth-general} and Lemma~\ref{lem:orbit} we obtain that the subgroup of $\aut X$ generated by the complete flows of $V_j$, $j=1,\dots, 5$, acts transitively on $X$.
\end{proof}
\bigskip

{\bf \em Part III. Existence of a compatible pair.}
\smallskip

\noindent We now find a pair of vector fields that are compatible pairs in the sense of Definition~\ref{def:compatible}.
\begin{proposition}\label{prop:compatible-AB}
The pair $(V_1,V_4)$ is compatible.  Its semicompatibility ideal can be taken to be all of $\holo(X_{m,n})$.
\end{proposition}

\begin{proof}
The kernels contain the polynomial subalgebras
\begin{equation}\label{eq:kernels}
 \cplx[t,z_0,z_2]\subset\Ker V_1,
 \qquad
 \cplx[t,z_1,z_3]\subset\Ker V_4.
\end{equation}
Every coordinate monomial factor can be written as
$\zeta^a z_0^b z_1^c z_2^d z_3^e$,
where $\zeta^az_0^bz_2^d\in\ker V_1$ and $z_1^c z_3^e\in\ker V_4 $
Hence, the linear span of $(\Ker V_1)(\Ker V_4)$ contains the full coordinate ring $\cplx[X]$.  By Lemma~\ref{lem:polynomial-dense}, its closure is all of $\holo(X)$.  Thus, $(V_1,V_4)$ is semicompatible with associated ideal $I=\holo(X)$. We now show that the pair is compatible. For that,
consider the holomorphic function $h(\zeta,z)=z_3$. We also know that
\[
 V_4(h)=V_4(z_3)=0,
 \qquad
 V_1(h)=V_1(z_3)=-z_2,
 \qquad
 V_1(z_2)=0.
\]
 Therefore,
\[
 h\in\Ker V_4,
 \qquad
 0\neq V_1(h)=-z_2\in\Ker V_1,
\]
and the pair is compatible.
\end{proof}

{\bf \em Part IV. There exists $p\in X$ such that $\{V_4|_p\}$ generates $T_pX$.}
\smallskip

\noindent Set
\begin{equation}\label{eq:p}
 p_0:=\left(1,0,0,\frac34,\frac12\right)\in\cplx^5.
\end{equation}
Clearly, $p_0\in X_{m,n}$.  
The tangent vectors corresponding to the complete fields are
\begin{align*}
V_1|_{p_0}&=\dfrac{\bdy}{\bdy z_1}|_{p_0}-\dfrac{3}{4}\dfrac{\bdy}{\bdy z_3}|_{p_0}\\
V_2|_{p_0}&= \dfrac{\bdy}{\bdy z_3}|_{p_0}-\dfrac{\bdy}{\bdy z_2}|_{p_0}\\
V_3|_{p_0} &= - \dfrac{3}{4}\dfrac{\bdy}{\bdy z_0}|_{p_0}\\
V_4|_{p_0}&= \dfrac{\bdy}{\bdy z_0}|_{p_0}\\
V_5|_{p_0} &= \dfrac{\bdy}{\bdy t}|_{p_0}-\dfrac{1}{4}\dfrac{\bdy}{\bdy z_2}|_{p_0}. \label{eq:p-E}
\end{align*}

\begin{lemma}\label{lem:basis-p0}
The vectors $V_1|_{p_0},V_2|_{p_0}, V_4|_{p_0},V_5|_{p_0}$ form a basis of $T_{p_0}X$.
\end{lemma}

\begin{proof}
Since 
\[
\dfrac{\bdy F_{m,n}}{\bdy z_2}|_{p_0}=1,
\] the projection
\[
 T_{p_0}X\longrightarrow\cplx^4_{\zeta,z_0,z_1,z_3}
\]
is an isomorphism.  Under this projection, the four vectors have coordinate columns
$\begin{pmatrix}
0\\
0\\
1\\
-3/4
\end{pmatrix}, 
\begin{pmatrix}
    0\\
    0\\
    0\\
    1
\end{pmatrix},
\begin{pmatrix}
    0\\
    1\\
    0\\
    0
\end{pmatrix}
$
and 
$\begin{pmatrix}
    1\\
    0\\
    0\\
    0
\end{pmatrix}$
in the order $(\zeta,z_0,z_1,z_3)$.
These are clearly a basis of $\cplx^4$.
\end{proof}

\begin{proposition}\label{prop:V4-generating}
The singleton $\{V_4|_{p_0}\}$ is a generating set of $T_{p_0}X$ in the sense of Definition~\ref{def:generating}.
\end{proposition}

\begin{proof}
The coordinate function $\varphi(z)=z_0$ belongs to
$ \Ker V_1\cap\Ker V_2\cap\Ker V_5.$
Therefore, the vector fields 
$
 \varphi V_1,
 \varphi V_2,
 \varphi V_5$
are complete.  They vanish at $p_0$ because $\varphi(p_0)=0$, so their flows lie in the isotropy subgroup $Aut(X)_{p_0}$.

For each $V_j, j=1,2,5$, Lemma~\ref{lem:replica-derivative} and the identity
\[
 V_4(\varphi)(p_0)=1
\]
give us that
\begin{equation}\label{eq:isotropy-action}
 d\Phi_t^{\varphi V_j}\big|_{p_0}\bigl(V_4|_{p_0}\bigr)
 =V_4|_{p_0}+tV_j|_{p_0}.
\end{equation}
Thus, the isotropy orbit of $V_4|_p$ contains
\[
 V_4|_{p_0}, V_4|_{p_0}+tV_j|_{p_0}, j=1,2,5.
 \]
By Lemma~\ref{lem:basis-p0}, these four vectors (taking $t\neq0$) form a basis: subtracting the first from the last three recovers $V_1|_{p_0},V_2|_{p_0},V_5|_{p_0}$.  Hence, $\{V_4|_{p_0}\}$ is a generating set.
\end{proof}

{\bf \em Part V. Completing the proof of Theorem~\ref{thm-density}}

\smallskip

%\begin{proof}[Proof of Theorem~\ref{thm-density}]
\noindent 
Smoothness and the Stein property were proved in Proposition~\ref{prop:smooth-general}.  The automorphism group acts transitively by Theorem~\ref{thm:transitivity}.  Proposition~\ref{prop:compatible-AB} gives a compatible pair $(V_1,V_4)$, and Proposition~\ref{prop:V4-generating} shows that its second member $V_4|_{p_0}$ is a generating set at $p_0$.  By Result~\ref{crit:leuenberger}, therefore, we conclude that $X$ has the density property.
%\end{proof}

\begin{remark}
The argument uses only one compatible pair.  The remaining vector fields are needed to establish transitivity and to produce enough isotropy derivatives to make the vector $V_4|_p$ generating.
\end{remark}

\section{The special members $X_{d-1,d}$}\label{sec-Kaliman}

In this section we identify the members $X_{d-1,d}$ with Kaliman modification of certain varieties.

For an integer $d\geq3$, put
\begin{equation}\label{eq:Pd}
 P_d(r,a,b,c):=r^d+ab+c^2
\end{equation}
and let
\[
 Z_d:=P_d^{-1}(0)\subset\cplx^4_{r,a,b,c}.
\]
The algebraic hypersurface $Z_d$ is irreducible and ts singular locus consists only of the origin. 
%since
%\[
% \nabla P_d=(d r^{d-1},b,a,2c).\]
The point $p=(0,1,0,0)$
lies in the smooth locus because $\dfrac{\bdy P_d}{\bdy b}(p)=1$.

Let $A_0:=\cplx[r,a,b,c]$ and $I_p=(r,a-1,b,c)$.  Let us denote $\Kalmod(\cplx^4,Z-d, p)$ as
$\mathcal{K}_4^d$. 
The Kaliman modification along $Z_d$ with center $p$ has the coordinate ring
\begin{equation}\label{eq:Bd}
 \mathcal B_d
 :=A_0[I_p/P_d]
 =\cplx[r,a,b,c]
 \left[
 \frac r{P_d},
 \frac{a-1}{P_d},
 \frac b{P_d},
 \frac c{P_d}
 \right]
 \subset A_0[P_d^{-1}].
\end{equation}

Introduce elements of $\mathcal B_d$ by
\begin{equation}\label{eq:mod-new-coordinates}
 \zeta=P_d,
 \qquad
 z_0=\frac r{P_d},
 \qquad
 z_1=\frac{a-1}{P_d},
 \qquad
 z_2=\frac b{P_d},
 \qquad
 z_3=\frac c{P_d}.
\end{equation}
Then
\begin{equation}\label{eq:mod-old-coordinates}
 r=\zeta z_0,
 \qquad
 a=1+\zeta z_1,
 \qquad
 b=\zeta z_2,
 \qquad
 c=\zeta z_3.
\end{equation}
Substitution into $\zeta=P_d$ gives
\begin{align*}
 \zeta
 &=\zeta^dz_0^d+(1+\zeta z_1)\zeta z_2+\zeta^2z_3^2\\
 &=\zeta\left(\zeta^{d-1}z_0^d+z_2(1+\zeta z_1)+\zeta z_3^2\right).
\end{align*}
The calculation occurs in the localization where $\zeta=P_d$ is invertible, so cancellation yields
\begin{equation}\label{eq:mod-relation}
 z_2(1+\zeta z_1)+\zeta z_3^2+\zeta^{d-1}z_0^d=1.
\end{equation}
This is precisely the equation of $X_{d-1,d}$. With the above notations we state:

\begin{proposition}\label{prop:coordinate-ring-isomorphism}
For every $d\geq3$, the assignments~\eqref{eq:mod-new-coordinates} induce an isomorphism
\begin{equation}\label{eq:ring-isomorphism}
 \frac{\cplx[\zeta,z_0,z_1,z_2,z_3]}
 {\bigl(z_2(1+\zeta z_1)+\zeta z_3^2+\zeta^{d-1}z_0^d-1\bigr)}
 \xrightarrow{\ \cong\ }
 \mathcal B_d.
\end{equation}
Consequently,
\begin{equation}\label{eq:general-kalmod}
 X_{d-1,d}\cong\Kalmod(\cplx^4,Z_d,\{p\}).
\end{equation}
\end{proposition}

\begin{proof}
The formulas in \eqref{eq:mod-new-coordinates} define a surjective homomorphism from the coordinate ring $\frac{\cplx[\zeta,z_0,z_1,z_2,z_3]}
 {\bigl(z_2(1+\zeta z_1)+\zeta z_3^2+\zeta^{d-1}z_0^d-1\bigr)}$ to $\mathcal B_d$. The elements $\zeta,z_0,z_1,z_2,z_3$ generate $\mathcal B_d$. 
By Proposition~\ref{prop:irreducible}, applied to $(m,n)=(d-1,d)$, the ring $\frac{\cplx[\zeta,z_0,z_1,z_2,z_3]}
 {\bigl(z_2(1+\zeta z_1)+\zeta z_3^2+\zeta^{d-1}z_0^d-1\bigr)}$ is a domain with Krull dimension four.  The ring $\mathcal B_d$ is a finitely generated domain contained in $A_0[P_d^{-1}]$ and contains $A_0$, hence,
\[
 \Frac(\mathcal B_d)=\Frac(A_0).
\]
Therefore, $\dim\mathcal B_d=4$.  The kernel of the surjection \[
\frac{\cplx[\zeta,z_0,z_1,z_2,z_3]}
 {\bigl(z_2(1+\zeta z_1)+\zeta z_3^2+\zeta^{d-1}z_0^d-1\bigr)}\to\mathcal B_d
 \]
 is a prime ideal.  If it were nonzero, it would have positive height and the quotient would have dimension at most three, which is a contradiction to $\dim\mathcal B_d=4$.  Hence, the kernel is trivial.
\end{proof}

The isomorphism can also be seen geometrically.  Define
\begin{equation}\label{eq:pi-modification}
 \pi_d:X_{d-1,d}\longrightarrow\cplx^4,
 \qquad
 \pi_d(\zeta,z_0,z_1,z_2,z_3):=(\zeta z_0,1+\zeta z_1,\zeta z_2,\zeta z_3).
\end{equation}
A direct computation using the defining equation gives
\begin{equation}\label{eq:P-pi=x}
 P_d\circ\pi_d=\zeta.
\end{equation}
Hence, $\pi_d$ restricts to an isomorphism
\begin{equation}\label{eq:off-exceptional-isomorphism}
 \{\zeta\neq0\}\subset X_{d-1,d}
 \xrightarrow{\ \cong\ }
 \cplx^4\setminus Z_d,
\end{equation}
with inverse
\[
 (r,a,b,c)\longmapsto
 \left(P_d,\frac r{P_d},\frac{a-1}{P_d},\frac b{P_d},\frac c{P_d}\right).
\]
The exceptional divisor is
\begin{equation}\label{eq:exceptional-divisor}
 E_d:=\{\zeta=0\}\cap X_{d-1,d}
 =\{\zeta=0,\ z_2=1\}\cong\cplx^3_{z_0,z_1,z_3},
\end{equation}
and $\pi_d(E_d)=\{p\}$.  Thus~\eqref{eq:pi-modification} has exactly the expected geometry of the affine modification.

\subsection{The case $d=7$}
In this case, the linear change of coordinates
\begin{equation}\label{eq:linear-brieskorn-change}
 a=u+iv,
 \qquad
 b=u-iv,
 \qquad
 c=s
\end{equation}
transforms
\[
 P_7=r^7+ab+c^2
\]
into
\begin{equation}\label{eq:brieskorn-polynomial}
 r^7+u^2+v^2+s^2.
\end{equation}
This is precisely the Brieskorn polynomial used by McLean~\cite[Section~3.1]{McLean2009} (see Section~\ref{sec-prelims} for a brief description).  Since the chosen $p$ is a nonzero smooth point of $Z_7$, Proposition~\ref{prop:coordinate-ring-isomorphism} proves Theorem~\ref{thm-identification}:
\begin{equation}\label{eq:X=K4}
 X_{6,7}\cong K_4.
\end{equation}

\section{Proof of Theorem~\ref{thm-main}}\label{sec-symplectic}

Parts $(i)$ and $(iii)$ of Theorem~\ref{thm-main} follows from McLean's work \cite{McLean2009}. Here we show how it follows with a brief description of the topology of $K_4$.

Consider the Brieskorn variety
\[
 Z_d=\{r^d+u^2+v^2+s^2=0\}\subset\cplx^4.
\]
The following result is proved by Brieskorn \cite{Brieskorn1966} (see also \cite[Theorem~3.1]{McLean2009} for $d=7$ and \cite[Theorem~3]{Uebele2016} for general case).
\begin{result}\label{res-topBrieskorn}
For $d\in\natu$, $d\geq 2$ and odd, $Z_d\cap S^7$ is homeomorphic to $S^5$
\end{result}
 The variety $Z_d$ is topologically the cone on its link and is therefore homeomorphic to $\rea^6$ ( see~\cite[Corollary~3.2]{McLean2009} for $d=7$). 

Let $p\in Z_d\setminus \{0\}$. By Result~\ref{res-kalmodcontractible}, we obtain that $\Kalmod(\cplx^4, Z_d, p)$ is contractible. In particular, $K_4$ is contractible.
Hence, by Proposition~\ref{prop:coordinate-ring-isomorphism}, $X_{d-1,d}$ is contractible.
McLean then applies the topology of Kaliman modifications (see \cite{Kaliman1994}): when the divisor and the center are acyclic topological manifolds, the modification of a contractible affine variety remains contractible~\cite[Lemma~2.27]{McLean2009}.  Thus, $K_4$ is contractible.

Finally, a contractible finite-type Stein manifold of complex dimension at least three is diffeomorphic to affine space by the $h$-cobordism argument recorded in~\cite[Corollary~2.30]{McLean2009}. 

\begin{result}\label{res-r8}
$\mathcal{K}_4^d=\Kalmod(\cplx^4, Z_d, p)$ is diffeomorphic to $\rea^8$.
\end{result}
In particular, $K_4$ is diffeomorphic to $\rea^8$.
Together with~\eqref{eq:X=K4}, this proves part~(i) of Theorem~\ref{thm-main}.

For a finite-type Stein manifold $M$ with trivialized canonical bundle, McLean considers the unital, graded-commutative symplectic homology ring over $\mathbb{Z}_2$. Let $i(M)$ denote the  number of idempotents in the symplectic homology $SH_*(M)$. We mention the following result from McLean's paper \cite[Section~3.2, especially the use of Theorems~2.31 and 7.7]{McLean2009}.
\begin{result}\label{res-idempotents}
The number of idempotents of $SH_*(K_4)$ satisfies:
\[
 1<i(K_4)<\infty.
\]
\end{result}
 Since a nonzero unital ring contains the distinct idempotents $0$ and $1$, and the zero ring has only one element, Result~\ref{res-idempotents} implies
\begin{equation}\label{eq:SHK4-nonzero}
 \SH(K_4;\F_2)\neq0.
\end{equation}

Proceeding similarly as in \cite{McLean2009}, we can prove that, for $d\geq 2$ an odd integer, 
\begin{equation}\label{eq-SH-nonzero}
   \SH(\mathcal{K}_4^d;\F_2)\neq0
\end{equation}

\begin{proposition}\label{prop:not-biholo-C4}
The affine fourfold $X_{6,7}$ is not biholomorphic to $\cplx^4$.
\end{proposition}

\begin{proof}
Suppose that $X_{6,7}$ were biholomorphic to $\cplx^4$.  Both are smooth affine varieties, hence finite-type Stein manifolds.  Proposition~\ref{prop:biholo-SH-invariance} would imply
\[
 \SH(X_{6,7};\F_2)
 \cong
 \SH(\cplx^4;\F_2)=0
\]
by~\eqref{eq:SH-C4}.  On the other hand, the identification~\eqref{eq:X=K4} and McLean's computation~\eqref{eq:SHK4-nonzero} give
\[
 \SH(X_{6,7};\F_2)\neq0.
\]
This contradiction proves the assertion.
\end{proof}

\begin{proof}[Proof of Theorem~\ref{thm-main}]
Part~(i) follows from~\eqref{eq:X=K4} and the fact that $K_4$ is diffeomorphic to $\cplx^4$.  Part~(ii) is the specialization $(m,n)=(6,7)$ of Theorem~\ref{thm-density}.  Part~(iii) is Proposition~\ref{prop:not-biholo-C4}.  These are exactly the three properties required to contradict the conclusion of the T\'oth--Varolin conjecture.
\end{proof}

\begin{proof}[Proof of Theorem~\ref{thm-counterfamily}]
By Result~\ref{res-r8}, $\mathcal{K}_4^d$ is diffeomorphic to $\rea^8$. Hence, by Proposition~\ref{prop:coordinate-ring-isomorphism}, $X_{d-1,d}$ is diffeomorphic to $\rea^8$. By Theorem~\ref{thm-density}, 
we obtain that $X_{d-1,d}$ has the density property.
    In view of equation~\eqref{eq-SH-nonzero} and Proposition~\ref{prop:coordinate-ring-isomorphism}, we see that for each $d\geq 2$ odd, $X_{d-1,d}$ is not biholomorphic to $\cplx^4$. Therefore, the fourfold $X_{d-1,d}$ satisfies all (i), (ii) and (iii) of Theorem~\ref{thm-main}.
\end{proof}

\noindent {\bf Acknowledgements.} The author is partially supported by an ARG MATRICS grant (ANRF/ARGM/2025/002958/MTR). The author thanks Souvik Goswami for discussions about the Kaliman modification.

\end{document}